\documentclass[a4paper, ariel, 11pt]{amsart}
\usepackage{amscd,amsthm,amsfonts,latexsym,amssymb}
\usepackage[dvips]{graphicx}

\usepackage{bez123,calc,curves,ebezier,epic,eepic,graphicx,multiply,rotating}

\newcommand{\ds}{\displaystyle}
\newtheorem{theorem}{Theorem} [section]
\newtheorem{lemma}[theorem]{Lemma}

\newtheorem{proposition}[theorem]{Proposition}

\newcommand{\inn}[1]{\langle#1\rangle} 

\newcommand{\dd}{{\mathrm d}}  

\newcommand{\RR}{{\mathbb R}}  
\newcommand{\CC}{{\mathbb C}}  

\newcommand{\ii}{{\rm i}}  

\newcommand{\ra}{\rightarrow}

\newcommand{\pa}{\partial}

\renewcommand{\phi}{\varphi}

\newcommand{\na}{\nabla}
\newcommand{\al}{\alpha}
\newcommand{\be}{\beta}

\newcommand{\om}{\omega}

\newcommand{\ov}{\overline} 

\renewcommand{\Re}{{\rm Re}\,} 
\renewcommand{\Im}{{\rm Im}\,} 

\begin{document}

\title{A class of analytic pairs of conjugate functions in dimension three}
\author{Paul Baird and Elsa Ghandour}

\address{Laboratoire de Math\'ematiques de Bretagne Atlantique UMR 6205 \\
Universit\'e de Brest \\
6 av.\ Victor Le Gorgeu -- CS 93837 \\
29238 Brest Cedex, France}

\email{Paul.Baird@univ-brest.fr, elsa.ghandour@hotmail.com}

\begin{abstract} 
We exploit an ansatz in order to construct power series expansions for pairs of conjugate functions defined on domains of Euclidean $3$--space.  Convergence properties of the resulting series are investigated.  Entire solutions which are not harmonic are found as well as a $2$-parameter family of examples which contains the Hopf map.      
\end{abstract}



\maketitle

\thispagestyle{empty}


\section{Introduction}   Following the first author and M. G. Eastwood \cite{Ba-Ea}, we will call a pair of differentiable real-valued functions $f$ and $g$ on a Riemannian manifold $M$ \emph{conjugate} if their gradients are everywhere orthogonal and of the same length: 
$$
\forall x \in M, \qquad ||\na f(x)|| = || \na g(x)|| \quad {\rm and}  \quad \inn{\na f(x), \na g(x)}= 0\,,
$$
 If we write $\phi = f + \ii g$, then these two equations can be summarized neatly by the equation 
$$
(\na \phi )^2 = 0\,;
$$
which says that the (complex) vector field $\na \phi $ (a section of the complexified tangent bundle $T^{\CC}M$) is isotropic at each point.  If $M = \RR^n$ with its standard metric, then this becomes the first order equation
\begin{equation} \label{sc-1}
\sum_{i = 1}^n \left(\frac{\pa \phi}{\pa x_i}\right)^2 = 0\,,
\end{equation}
where $(x_1, \ldots , x_n)$ are standard coordinates for $\RR^n$.   When $n =3$, it is shown in \cite{Ba-Ea-0} that a solution extends locally to an integrable Hermitian structure on a domain of $\RR^4$.  In the case when $n = 2$, on writing $\ds \frac{\pa}{\pa z} = \frac{1}{2} \left( \frac{\pa}{\pa x_1} - \ii \frac{\pa}{\pa x_2}\right)$ and   $\ds \frac{\pa}{\pa \ov{z}} = \frac{1}{2} \left( \frac{\pa}{\pa x_1} + \ii \frac{\pa}{\pa x_2}\right)$, we have
$$
\frac{\pa\phi}{\pa z}\frac{\pa \phi}{\pa \ov{z}} = 0\,,
$$
whose solutions locally correspond to holomorphic or anti-holomorphic functions.  Thus the condition we are discussing is a natural generalization of holomorphicity in the plane.   However, for $n \geq 3$, we cannot in general separate \eqref{sc-1} into a product of linear conditions. 

For a differentiable mapping $\phi : U \subset \RR^n \ra \CC$, where the domain and codomain are endowed with their canonical metrics, the condition \eqref{sc-1} is equivalent to the requirement that $\phi$ be \emph{semi-conformal}, that is, for each $x \in U$, the derivative $\dd \phi_x$ is either identically zero, or the restriction to the horizontal space $\dd \phi_x\vert_{(\ker \dd \phi_x)^{\bot}} : (\ker \dd \phi_x)^{\bot} \ra T_{\phi (x)}\CC$ is conformal and surjective.  This is one of two conditions that characterizes a \emph{harmonic morphism}, which, by definition is a mapping which pulls back local harmonic functions to harmonic functions; the other condition being the harmonicity of $\phi$.  This characterization of harmonic morphisms between Riemannian manifolds was established independently in 1978 by B. Fuglede \cite{Fu}, in 1979 by Y. Ishihara \cite{Is}, and in the context of mappings which preserve Brownian motion, in 1979 by A. Bernard, E. Campbell and A. M. Davie \cite{Be-Ca-Da}, although the notion dates back to 1967 in the more general context of Brelot harmonic spaces \cite{Br}.

In the case when a semi-conformal map $\phi$ takes values in a surface, then by conformal invariance, on choosing isothermal coordinates, we may suppose that locally it takes values in the complex plane and so can be written in the form $\phi = f + \ii g$ for two real-valued conjugate functions.  In the article \cite{Ba-Ea}, the question was asked as to whether there exists a differential condition on a function $f : U \ra \RR$ ($U$ open in $\RR^n$) in order that it admit a conjugate.  For $n = 2$, the answer is well known, since if $f$ admits a conjugate $g$, then $f + \ii g$ is $\pm$--holomorphic and in particular \emph{harmonic}.  Conversely, if $f$ is harmonic on a simply connected domain, then it admits a conjugate.  Indeed, on writing $f_1 = \pa f / \pa x_1$ etc., the $1$--form $\om = -f_2 \dd x_1 + f_1 \dd x_2$ is closed and so there exists a function $g$ such that $\om = \dd g$ and then $g$ is conjugate to $f$.  In the case when $n = 3$, the question becomes remarkably difficult to answer. 
It was shown in \cite{Ba-Ea} that there is a 2nd order differential inequality and three 3rd order equations which characterize such functions, all conformally invariant. 

The complexity of the equations makes it hard to find concrete examples.  Functions with spherical and cylindrical symmetry were characterized in \cite{Ba-Ea}, where also an ansatz was applied to obtain examples from solutions to a first order equation in two variables.  However, only a handful of simple solutions were found by this method.  It is our aim in this article to explore this ansatz more fully and to characterize the analytic solutions.

 Any semi-conformal polynomial map is automatically harmonic \cite{Ab-Ba-Br}, and any entire harmonic morphism on $\RR^3$ with values in a surface is an orthogonal projection followed by a weakly conformal map \cite{Ba-Wo-0}.  In \S\ref{sec:one}, by the ansatz, we construct entire analytic semi-conformal maps on $\RR^3$ with values in a surface which are not harmonic (and so not simple projections).  In \S\ref{sec:two} we find a two parameter family of analytic semi-conformal maps which contains the Hopf map as a special case.  

\section{An ansatz for solutions}  Consider $\RR^3$ with coordinates $(x, y, z)$ and define the variable $\ds u = \tfrac{x^2+y^2}{2}$.  Suppose that $\phi = f + \ii g$ can be expressed in the form $\phi (x,y,z) = (x + \ii y)u^{-q} \psi (u,z)$, for some \emph{complex}-valued function $\psi$ analytic in a neighbourhood of $(u,z) = (0,0)$ and for some integer $q\geq 0$.  We suppose also that $\psi$ doesn't contain a factor of $u$ so $q$ is the smallest integer which allows $\phi$ to be written in this way.  Write $\phi_x$ for $\pa \phi / \pa x$, and so on.  Then 
$$
\phi_x{}^2+\phi_y{}^2+\phi_z{}^2 = 2(x+ \ii y)^2u^{-2q-1} \{ q(q-1)\psi^2 + u(1-2q)\psi \psi_u +u^2 \psi_u{}^2 + \tfrac{1}{2}u \psi_z{}^2\}\,,
$$
so that $\phi$ is semi-conformal (equivalently $f$ and $g$ are conjugate) if and only if 
\begin{equation} \label{ans-1}
q(q-1)\psi^2 + u(1-2q)\psi \psi_u +u^2 \psi_u{}^2 + \tfrac{1}{2}u \psi_z{}^2=0\,.
\end{equation}
We have thus reduced the problem to one involving just two variables.  
Note that any solution $\phi$ to \eqref{sc-1} can be multiplied by a constant to give another solution, so that if $\psi (0,0) \neq 0$, we may suppose that $\psi (0,0) = 1$.  Furthermore, if $\psi (0,0) \neq 0$, then on setting $u=0$ in \eqref{ans-1}, we see that we must have $q = 0$ or $1$.  In what follows we will assume that this is the case and so we are led to consider the two equations
\begin{eqnarray}
\psi \psi_u +u \psi_u{}^2 + \tfrac{1}{2} \psi_z{}^2 & = & 0  \qquad (q=0) \label{q0} \\
- \psi \psi_u +u \psi_u{}^2 + \tfrac{1}{2} \psi_z{}^2 & = & 0  \qquad (q=1) \label{q1}
\end{eqnarray}

A particular solution of \eqref{q1} is given by the polynomial 
\begin{equation} \label{Hopf} 
\psi (u,z) = 1 - 2u - z^2 - 2\ii z,
\end{equation}
 which corresponds to the Hopf map $\phi : \RR^3 \ra \CC \cup \{ \infty\}$ whose fibres are the circles of Villarceau.  This map may be constructed by taking the Hopf fibration $S^3 \ra S^2$ and pre-composing and post-composing by the inverse of stereographic projection and stereographic projection, respectively.

 In what follows, we will show that the free parameters in the general analytic solution to \eqref{q0} and \eqref{q1} are the values at the origin of the successive derivatives $\psi_z, \psi_{zz}, \psi_{zzz}, \ldots$.    We will then set up general recurrence relations that enable the general solution to be computed in principle, however, to find explicit expressions still remains difficult.  We investigate in detail the solution to \eqref{q0} corresponding to the given data at the origin: $\psi_z(\vec{0}) = 0, \psi_{z}(\vec{0}) = c$ with all other derivatives $\ds \frac{\pa^{\ell}\psi}{\pa z^{\ell}} (\vec{0}) = 0$ for $\ell \geq 2$, where $c$ is an arbitrary complex parameter.  We then investigate the solution to \eqref{q1} corresponding to the given data at the origin: $\psi_z(\vec{0}) = 0, \psi_{z}(\vec{0}) = c_1, \psi_{zz}(\vec{0}) = c_2$ with all other derivatives $\ds \frac{\pa^{\ell}\psi}{\pa z^{\ell}} (\vec{0}) = 0$ for $\ell \geq 3$, where $c_1$ and $c_2$ are arbitrary complex parameters.  Amongst this family of solution we find the Hopf map discussed in the previous paragraph.  In particular the Hopf map belongs to a $2$-complex parameter family of semi-conformal mappings.  

\section{Algebraic preliminaries} Consider a complex-valued analytic function $\psi (u,z)$ expressed in the form
\begin{equation} \label{solution}
\psi (u,z) = \sum_{k,\ell = 0}^{\infty} a_{k,\ell}\,u^kz^{\ell}\,, 
\end{equation}
for constants $a_{k,\ell}$.  For background on analytic functions in several variables we refer the reader to \S 9 of the classical text of J. Dieudonn\'e \cite{Di}.  In particular, if the series converges absolutely on an open set, then it determines a well-defined infinitely differentiable function on that set.  We are interested in domains of convergence in a neighbourhood of $(0,0)$.  For a solution to be interpreted as a semi-conformal map on a domain of $\RR^3$, we require that $u\geq 0$.  An example, given in \cite{Ba-Ea}, is the function
$$
\psi (u,z) = \frac{be^{cz}e^{\sqrt{1-2c^2u}}}{1 + \sqrt{1 - 2c^2u}}\,,
$$
where $b$ and $c$ are arbitrary constants.  This solution has product form and is analytic in a neighbourhood of the origin.  

More generally the Cauchy-Kowalewski Theorem can be applied to deduce the existence of analytic solutions to \eqref{q0} and \eqref{q1}.  Specifically, if we write the equations in the form:
\begin{equation}\label{CK}
\psi_z = \left(  \pm 2 \psi \psi_u -2 u \psi_u{}^2\right)^{1/2},
\end{equation}
then, given the boundary condition $\psi (u,0) = \xi (u)$ for some analytic function $\xi$ defined in a neighbourhood of $u = 0$ with $\xi (0)\xi^{\prime}(0) \neq 0$, there is a unique analytic solution $\psi (u,z)$ extending this boundary value in a neighbourhood of the point $(u,z) = (0,0)$ \cite{Ko}.

In order to find  recursive relations between the constants $a_{k,\ell}$, we will compute successive derivates at the origin:
\begin{equation} \label{der-eq}
\frac{\pa^n}{\pa u^k\pa z^{\ell}}\left\{ \pm \psi \psi_u + u \psi_u{}^2 + \tfrac{1}{2} \psi_z{}^2\right\} \vert_{(u,z) = 0}\quad (k+\ell = n)
\end{equation} 
and equate these to zero.   In what follows, for ease of notation, we set
$$
\psi_{k,\ell} = \frac{\pa^{k+ \ell}\psi}{\pa u^k\, \pa z^{\ell}}(0,0)\,.
$$
Thus $\ds a_{k,\ell} = \frac{\psi_{k, \ell}}{k!\ell !}$.  The fundamental identity we exploit is given by the following proposition, whose proof is given in Appendix \ref{app:calcs}. 

\begin{proposition} \label{prop:fund} A necessary and sufficient condition that \eqref{der-eq} vanish is that for each $k \geq 1$ and $\ell \geq 0$, we have
\begin{eqnarray} 
& \ds \sum_{j = 0}^{\ell} \sum_{i = 0}^k (k- i \pm 1)\left( \begin{array}{c} \ell \\ j \end{array} \right) \left( \begin{array}{c} k \\ i \end{array} \right) \psi_{k-i, \ell - j}\, \psi_{i+1,j} &  \label{eq1} \\
& \ds \qquad \qquad +  \sum_{j = 0}^{\ell} \sum_{i = 0}^{k-1} \left( \begin{array}{c} \ell \\ j \end{array} \right)\left( \begin{array}{c} k-1 \\ i \end{array} \right)\psi_{k-i-1, \ell - j + 1}\, \psi_{i+1, j+1} &  = 0\,. \nonumber
\end{eqnarray}
\end{proposition}

Note that the choice of sign in \eqref{der-eq} is incorporated into the first bracket of  \eqref{eq1}.

\begin{theorem} \label{th:free-parameters} Solutions to {\rm \eqref{q0}} and {\rm \eqref{q1}} of the form {\rm \eqref{solution}} with $ \psi_{0,0}\neq 0$ and $\psi_{0,1} \neq 0$ are uniquely characterized by the given data: $\psi_{0, \ell}$,  for $\ell = 0,1,2, \ldots$.
\end{theorem}

Note that from either \eqref{q0} or \eqref{q1}, since $\psi_{0,0}, \psi_{0,1} \neq 0$ it follows that $\psi\psi_u \neq 0$ at the origin $(u,z) = (0,0)$ and so the term in the bracket on the right-hand side of \eqref{CK} is non-zero and we can expect to find analytic solutions in a neighbourhood of $(u,z) = (0,0)$.  As previously noted,  provided that $\psi_{0,0} \neq 0$, we may suppose that $\psi_{0,0} = 1$.  

\begin{proof}  To prove the theorem, we apply Proposition \ref{prop:fund} to isolate the highest order derivates of $\psi$ (of order $n+1$) in \eqref{der-eq}.  These occur in the first sum of the proposition when $i = k$ and $j = \ell$ ($k+\ell = n$), and in the second sum when $i = k-1$ and $j = \ell$.  We deduce that for each $n \geq 1$, and each $k = 0,1,\ldots , n$, we have 
$$
\pm \psi_{0,0} \psi_{k+1,n-k} + \psi_{0,1} \psi_{k,n-k+1} + R(n,k) = 0
$$
where the term $R(n,k)$ involves derivatives of $\psi$ of order $\leq n$ (for the case $k = 0$ it suffices to equate  \eqref{der-eq} to zero with $k = 0$, $\ell = n$ and set $u = 0$).  Recall that by hypothesis, $\psi_{0,1} \neq 0$.   Combine successive equations:
$$
\begin{array}{lcc}
\pm \psi_{0,0} \psi_{k+1,n-k} + \psi_{0,1} \psi_{k,n-k+1} + R(n,k) & = & 0 \\
\pm \psi_{0,0} \psi_{k+2,n-k-1} + \psi_{0,1} \psi_{k+1,n-k} + R(n,k+1) & = & 0
\end{array}
$$
in the following way.  Begin with $k=0$; multiply the first equation by $\psi_{0, 1}$, the second by $\psi_{0,0}$ and then add or subtract, according to the sign.  Now proceed to $k=1$, and so on, to deduce the identity
\begin{equation} \label{ders-psi}
\pm (\psi_{0,1})^{k+1} \psi_{0,n+1} + (-1)^k(\psi_{0,0})^{k+1} \psi_{k+1,n-k} + \sum_{j = 0}^k (-1)^j(\psi_{0,1})^{k-j} (\psi_{0,0})^j R(n,j) = 0.
\end{equation}
for each $k = 0, \ldots , n$.  
We can now apply induction on $n$.  Thus we suppose that for $k + \ell = n$, all $\psi_{k,\ell}$ are determined by $\psi_{0,0},\psi_{0,1}, \ldots , \psi_{0,n}$.   To begin the induction, from \eqref{ans-1}, we see that $\psi_{1,0}$ is determined by $\psi_{0,0}$ and $\psi_{0,1}$.  But now equation \eqref{ders-psi} shows that for each $k = 0, \ldots , n$ the derivative at the origin $\psi_{k+1,n-k}$ is determined by $\psi_{0,0},\psi_{0,1}, \ldots , \psi_{0,n}, \psi_{0,n+1}$ and the induction is complete.  

\end{proof}


\section{Particular solutions: a $1$-parameter family} \label{sec:one}  Our objective is to determine particular solutions of \eqref{q0} and \eqref{q1} of the form \eqref{solution}.  In some cases, we can recognize the power series and express the solutions in closed form.  As a first case, we take the given data:
\begin{equation} \label{data}
\psi_{0,0} = 1, \quad \psi_{0,1} = c, \quad \psi_{0,\ell} = 0 \ {\rm for} \ \ell \geq 2
\end{equation}
for some arbitrary complex constant $c$.  By Theorem \ref{th:free-parameters}, this uniquely characterizes the solution.  Furthermore, by \eqref{q0} and \eqref{q1}, at the point $(u,z)= (0,0)$, we have $\psi \psi_u = - \frac{1}{2}c^2$ and $\psi \psi_u =  \frac{1}{2}c^2$, respectively.  In particular, if $c \neq 0$, the expression inside the bracket on the right-hand side of \eqref{CK} is non-zero and by the Cauchy-Kowalewksi Theorem we can anticipate a convergent solution in a neighbourhood of the origin.     

\begin{theorem} \label{th:solution-1}  For the given data {\rm \eqref{data}}, the solution $\psi (u,z) = \sum_{k,\ell = 0}^{\infty} a_{k,\ell} u^kz^{\ell}$ to {\rm \eqref{q0}} has the form
\begin{equation} \label{eq2}
\psi (u,z) = 1 + cz + \sum_{k = 1}^{\infty} \sum_{\ell = 0}^{\infty} (-1)^{\ell + 1} c^{\ell + 2k} \left( \begin{array}{c} \ell + 2k - 2 \\ \ell  \end{array} \right) \frac{3^{k-1}(2k-2)!}{2^{k-1}(k+1)!(k-1)!}u^kz^{\ell}\,.
\end{equation}
The solution to {\rm \eqref{q1}} has the form
\begin{equation} \label{eq3}
\psi (u,z) = 1 + cz + \sum_{k = 1}^{\infty} \sum_{\ell = 0}^{\infty} (-1)^{\ell + k-1} c^{\ell + 2k} \left( \begin{array}{c} \ell + 2k - 2 \\ \ell  \end{array} \right) \frac{(2k-2)!}{2^{k}k!(k-1)!}u^kz^{\ell}\,.
\end{equation}
Furthermore, these power series expressions converge absolutely for all $(u,z)$ in a neighbourhood of the origin $(0,0)$.  
\end{theorem}

\begin{proof}  By ad hoc methods, one can deduce that the series has the form:
$$
\psi (u,z) = 1 + cz + \sum_{k = 1}^{\infty} \sum_{\ell = 0}^{\infty} (-1)^{\ell + 1} c^{\ell + 2k} \left( \begin{array}{c} \ell + 2k - 2 \\ \ell  \end{array} \right) f(k)\,u^kz^{\ell}\,,
$$
for some function $f$ which depends only on $k$.  Then, on setting $f(0) = -1$, we have
$$ 
\begin{array}{lcl}
 \psi\vert_{z=0} & = &  - \sum_{k = 0}^{\infty} c^{2k} f(k) u^k \\
 \psi_u\vert{z=0} & = & - \sum_{k = 0}^{\infty} c^{2k+2} (k+1) f(k+1) u^k \\
 \psi_z\vert{z=0} & = & c\sum_{k = 0}^{\infty} c^{2k}(1+2k-2) f(k) u^k
\end{array}
$$

Recall the Cauchy product of two power series: 
$$
\left( \sum_{i = 0}^{\infty} a_ix^i \right) \left( \sum_{j = 0}^{\infty} b_j x^i\right) = \sum_{k = 0}^{\infty} c_kx^k\,,
$$
where $c_k = \sum_{m = 0}^ka_mb_{k-m}$.  We thereby obtain at $z=0$:
$$
\begin{array}{lcl}
\psi \psi_u & = & \sum_{k = 0}^{\infty}c^{2k+2} \left( \sum_{m = 0}^k(k-m+1)f(m)f(k-m+1)\right) u^k \\
\psi_u{}^2 & = & \sum_{k = 0}^{\infty} c^{2k+4} \left( \sum_{m = 0}^k(m+1)(k - m+1)f(m+1)f(k - m+1) \right) u^k \\
\psi_z{}^2 & = & \sum_{k = 0}^{\infty} c^{2k+2} \left( \sum_{m = 0}^k(2m-1)(2k-2m-1)f(m)f(k-m)\right) u^k
\end{array}
$$
In particular (at $z=0$)
\begin{eqnarray*}
u\psi_u{}^2 & = & \sum_{k = 0}^{\infty} c^{2k+4} \left( \sum_{m = 0}^k(m+1)(k - m+1)f(m+1)f(k - m+1) \right) u^{k+1} \\
&  = & \sum_{k = 1}^{\infty}c^{2k+2} \left( \sum_{m = 0}^{k-1}(m+1)(k-m)f(m+1)f(k-m)\right) u^k\,.
\end{eqnarray*}
First consider \eqref{q0}.  On combining the various terms and equating the different powers $k$ of $u$ to zero, we find that for $k = 0$ we have $f(1) = 1/2$, and for $k \geq 1$, we have
\begin{eqnarray*}
\sum_{m = 0}^k\Big\{ (k-m+1)f(m)f(k-m+1) & + & (m+1)(k-m)f(m+1)f(k-m) \\
 & &   + \frac{1}{2}(2m-1)(2k-2m-1)f(m)f(k-m)\Big\} =0
\end{eqnarray*}
On separating the $m = 0$ part of the first term and replacing $m$ by $m-1$ in the second, we obtain the recurrence:
\begin{eqnarray}
(k+1)f(k+1) &  =  & \sum_{m = 0}^k\Big\{ (m+2)(k - m)f(m+1)f(k-m) \label{rec1}\\
 & & \qquad \qquad + \frac{1}{2}(2m-1)(2k-2m - 1)f(m)f(k-m)\Big\}\,. \nonumber
\end{eqnarray}
In Appendix \ref{app:calcs2} we show that this has solution  
\begin{equation} \label{sol-rec}
f(k) =  \frac{3^{k-1}(2k-2)!}{2^{k-1}(k+1)!(k-1)!} = \frac{1}{k(k+1)}\left( \begin{array}{c} 2k-2 \\ k-1 \end{array} \right) \left(\frac{3}{2}\right)^{k-1}\,.
\end{equation} 
for $k \geq 1$.  Similarly, \eqref{q1} leads to the recurrence
\begin{eqnarray}
(k+1)f(k+1) &  =  & - \sum_{m = 0}^k\Big\{ m(k - m)f(m+1)f(k-m) \label{rec2}\\
 & & \qquad \qquad + \frac{1}{2}(2m-1)(2k-2m - 1)f(m)f(k-m)\Big\}\,, \nonumber
\end{eqnarray}
with solution 
$$
f(k) = (-1)^k\frac{(2k-2)!}{2^kk!(k-1)!}
$$
for $k \geq 1$.

For the convergence, it suffices to set $z=0$ in \eqref{eq2}, to obtain the series in the single variable $u$:
$$
\xi (u) = 1 - \sum_{k = 1}^{\infty}  c^{2k}  \frac{3^{k-1}(2k-2)!}{2^{k-1}(k+1)!(k-1)!}u^k\,.
$$
By the test d'Alembert, this converges absolutely for $|u|<1/6|c|^2$ and by the Cauchy-Kowalewski theorem we have convergence of the extension determined by \eqref{CK} into the $(u,z)$--plane; this extension in particular contains a neighbourhood of the origin.   
\end{proof}

We can find closed expressions for the above power series as follows.  
Recall Newton's generalized binomial formula:
\begin{equation} \label{nbf}
\frac{1}{(1-t)^{r}} = \sum_{n = 0}^{\infty} \left( \begin{array}{c} n +r-1  \\ r -1 \end{array} \right) t^n\,,
\end{equation}
where the series on the right converges absolutely for $|t|<1$.   We can re-write \eqref{eq2} as
$$
\psi (u,z) = 1 + cz - \sum_{k = 1}^{\infty} \left(\sum_{\ell = 0}^{\infty}  \left( \begin{array}{c} \ell + 2k - 2 \\ 2k-2  \end{array} \right)(-cz)^{\ell}\right)  \frac{2(2k-2)!}{3(k+1)!(k-1)!}\left(\frac{3c^{2}u}{2}\right)^k\,.
$$
By Newton's formula
$$
\sum_{\ell = 0}^{\infty}  \left( \begin{array}{c} \ell + 2k - 2 \\ 2k-2  \end{array} \right) (-cz)^{\ell} = \frac{1}{(1+cz)^{2k-1}}
$$
(with the series absolutely convergent for $|cz|<1$), so that
$$
\psi (u,z) = 1 + cz - \tfrac{2}{3}(1+cz) \sum_{k = 1}^{\infty}   \frac{(2k-2)!}{(k+1)!(k-1)!}\left(\frac{3c^{2}u}{2(1+cz)^2}\right)^k\,.
$$
We can now be more specific about the domain of convergence.  Once more, by the d'Alembert test we require that:
$$
\left|\frac{3c^{2}u}{2(1+cz)^2}\right| < 1/4 \quad \Rightarrow \quad |u| < |1+cz|^2/6|c|^2\,,
$$
(with $|z|<1/|c|$) which, for $z=0$ coincides with the interval of convergence established in the above theorem. 

Consider the sum
\begin{equation}\label{sum}
S(t):= \sum_{k = 1}^{\infty} \frac{1}{k(k+1)}\left( \begin{array}{c} 2k-2 \\ k-1 \end{array} \right) t^k\,.
\end{equation}
Then
$$
\frac{\dd^2}{\dd t^2}(tS(t)) = \sum_{k = 1}^{\infty} \left( \begin{array}{c} 2k-2 \\ k-1 \end{array} \right) t^{k-1} = \sum_{k = 0}^{\infty}\left( \begin{array}{c} 2k\\ k \end{array} \right) t^{k} = \frac{1}{\sqrt{1-4t}}\,. 
$$
The latter equality is once more a consequence of Newton's binomial formula and the series converges absolutely for $|t|<1/4$.  We can now integrate this expression to obtain
$$
S(t) = \frac{1}{12t}(1-4t)^{3/2} + c_1 + \frac{c_2}{t}\,,
$$
where the constants are chosed to agree with the form \eqref{sum}, explicitly
$$
S(t) = \frac{1}{12t}(1-4t)^{3/2} - \frac{1}{12t} + \frac{1}{2} \,.
$$
On substituting $t = 3c^{2}u/2(1+cz)^2$, this gives the solution in the form
$$
\psi (u,z)  = \frac{2}{3}(1 + cz) - \frac{\{ (1+cz)^2 - 6c^2u\}^{3/2}}{3^3c^2u} + \frac{(1+cz)^3}{3^3c^2u} \,.
$$
The singularity when $u = 0$ is removable, so this represents an entire analytic solution provided that $t$ is defined for all $(u,z)\in \RR^2$ ($u\geq 0$) and that $  (1+cz)^2 - 6c^2u \neq 0$ for all $(u,z) \in \RR^2$ ($u\geq 0$).  A judicious choice of the constant ensures that this is the case, for example $c =\ii$. 

Any mapping of the form $\phi (x,y,z) = (x + \ii y)u^{-q} \psi (u,z)$ is harmonic if and only if 
$$
q(q-1) \psi - 2(q-1)u\psi_u + u^2 \psi_{uu} + \tfrac{1}{2}u\psi_{zz} = 0.
$$
 It is readily checked from this formula that the above solution is not harmonic.  It thus gives an example of an entire analytic semi-conformal map which is not harmonic.  By comparison, the  mapping $\phi (x,y,z) = \sqrt{x^2+y^2} + \ii z$ is defined and continuous everywhere on $\RR^3$, but it is only semi-conformal and analytic at all points where $x^2+y^2 \neq 0$.       

A similar treatment of equation \eqref{eq3} yields its closed form.  It is convenient to multiply the solution by $2$ to give the simple expression:	
$$
\psi (u,z) = 1+cz +  \sqrt{2c^2u + (1+cz)^2}\,.
$$
Once more, one can choose the constant $c$ so that the term inside the square root never vanishes, however, we recall that $\phi = (x+\ii y)\psi (u,z)/u$, so that $\phi$ has a singularity along the $z$-axis.  In fact one can now check that in this case the fibres of $\phi$ are straight lines and that $\phi$ is harmonic, and so a harmonic morphism.

\section{Solutions with two free parameters} \label{sec:two}  In this section, we concentrate on \eqref{q1} with two free parameters, which has the Hopf map as a particular solution.  Consider the given data:
\begin{equation} \label{data2}
\psi_{0,0} = 1, \quad \psi_{0,1} = c_1, \quad \psi_{0,2} = c_2 \quad \psi_{0,\ell} = 0 \ {\rm for} \ \ell \geq 3
\end{equation}
for arbitrary complex constants $c_1$ and $c_2$.  By Theorem \ref{th:free-parameters}, this uniquely characterizes the solution.  It is convenient to express $c_1$ and $c_2$ in the form
$$
c_1 = \al + \be \qquad c_2 = 2\al \be
$$
for constants $\al$ and $\be$ satisfying $\al + \be \neq 0$.  The Hopf map is then determined by $\al = \be = -\ii$.  
For each $k \geq 1$, we will generate a recurrence relation in $\ell$ for $\psi_{k,\ell}$.  We begin with $k = 1$.   Note that from \eqref{q1}, we have
$$
\psi_{1,0} = \tfrac{1}{2}(\al + \be )^2\,.
$$

Write $\pa_z^{\ell}$ for $\pa^{\ell}/\pa z^{\ell}$.  Then, from \eqref{data2},
$$
\pa_z^{\ell}(\psi \psi_u) = \sum_{s=0}^{2} \left( \begin{array}{c} \ell \\ s \end{array} \right) \pa_z^{s}\psi \, \pa_z^{\ell-s} \psi_u = \psi_{1, \ell} + \ell \psi_{0,1}\psi_{1, \ell - 1} + \tfrac{1}{2} \ell (\ell - 1) \psi_{0,2}\psi_{1, \ell - 2}\,,
$$
and
$$
\pa_z^{\ell}(\tfrac{1}{2}\psi_z{}^2) = \pa_z^{\ell -1}(\psi_z\psi_{zz})\,,
$$
which vanishes for $\ell \geq 3$.  We deduce that
$$ 
\psi_{1,1} = - \tfrac{1}{2}(\al + \be )(\al - \be )^2, \quad \psi_{1,2} = (\al - \be )^2(\al^2 + \al \be + \be^2)
$$
and for $\ell \geq 3$
\begin{eqnarray*}
\psi_{1, \ell} & = & - \ell (\al + \be ) \psi_{1, \ell - 1} - \ell (\ell - 1)\al \be \psi_{1, \ell - 2} \\
& \Rightarrow &  \psi_{1, \ell} + \ell \al \psi_{1, \ell - 1} = - \be \ell (\psi_{1, \ell - 1} + \al (\ell - 1) \psi_{1, \ell - 2})\,.
\end{eqnarray*}
On writing $u_{\ell} = \psi_{1, \ell} + \al \ell \psi_{1, \ell - 1}$ this becomes
$$
u_{\ell} = - \be \ell u_{\ell - 1}
$$
which is solved by 
$$
u_{\ell} = \frac{(-1)^{\ell}}{2} \ell !\, \be^{\ell - 2} u_2 \quad {\rm where} \quad u_2 = \psi_{1, 2} + 2 \al \psi_{1,1} = \be^2(\al - \be )^2;
$$
explicitly,
$$
u_{\ell} = \frac{(-1)^{\ell}}{2} \ell ! \,(\al - \be )^2 \be^{\ell} \quad \Rightarrow \quad \psi_{1, \ell} + \al \ell \psi_{\ell - 1} = \frac{(-1)^{\ell}}{2} \ell !\, (\al - \be )^2 \be^{\ell}
$$
with solution
$$
\psi_{1, \ell} = \frac{(-1)^{\ell}}{2} \ell ! \, (\al - \be )^2 \sum_{r = 0}^{\ell} \al^{\ell - r}\be^r = \frac{(-1)^{\ell}}{2} \ell ! \, (\al - \be ) (\al^{\ell + 1} - \be^{\ell + 1} )
$$
valid for all $\ell \geq 1$.  Note that, by the way we have chosen the constants, the expression is both homogeneous and symmetric in $\al$ and $\be$.   It is instructive to also work through the case $k = 2$.  
We obtain the recurrence
$$
\begin{array}{l}
\psi_{2, \ell} + \ell (\al + \be ) \psi_{2, \ell - 1} + \ell (\ell - 1) \al \be \psi_{2, \ell - 2}  = \\
\ds  \qquad  \frac{(-1)^{\ell + 1}}{2} \ell ! \, (\al - \be )^2 \left\{ (\ell + 1) \al^{\ell + 2} + 2 \al^{\ell + 1} \be + 2 \al^{\ell} \be^2 + \cdots + 2 \al \be^{\ell + 1} + (\ell + 1) \be^{\ell + 2} \right\}
\end{array}
$$
which, on setting $v_{\ell} = \psi_{2, \ell} + \ell \al \psi_{2, \ell - 1}$, becomes
$$
\begin{array}{l}
v_{\ell} + \ell \be v_{\ell - 1} = \\
\ds  \qquad  \frac{(-1)^{\ell + 1}}{2} \ell ! \, (\al - \be )^2 \left\{ (\ell + 1) \al^{\ell + 2} + 2 \al^{\ell + 1} \be + 2 \al^{\ell} \be^2 + \cdots + 2 \al \be^{\ell + 1} + (\ell + 1) \be^{\ell + 2} \right\}
\end{array}
$$
with solution 
$$
\begin{array}{l}
\psi_{2, \ell} =  \\
\ds  \qquad  \frac{(-1)^{\ell + 1}}{2} \ell ! \, (\al - \be ) \Big\{ \tfrac{1}{2} (\ell + 1) (\ell + 2)\al^{\ell + 3} + (\ell + 1)\al^{\ell + 2} \be + (\ell - 1) \al^{\ell + 1} \be^2 + \cdots \\
\qquad \qquad \cdots + (- \ell + 3)\al^3\be^{\ell} + (- \ell + 1) \al^2 \be^{\ell + 1} - (\ell + 1) \al \be^{\ell + 2} - \frac{1}{2} (\ell + 1)(\ell + 2) \be^{\ell + 3}\Big\}
\end{array}
$$
valid for all $\ell \geq 0$.  

Proceeding recursively, one deduces that
$$
\psi_{k, \ell} = \frac{(-1)^{k + \ell + 1}}{2} \ell ! \, (\al - \be )^2 g_{k,\ell}(\al , \be )\,,
$$
where $g_{k, \ell}(\al , \be )$ is a symmetric polynomial in $\al$ and $\be$ homogeneous of degree $\ell + 2(k -1)$.  Although, we do not have the general expression for $g_{k, \ell}(\al , \be )$, we can obtain one when $\ell = 0$.

\begin{theorem} \label{thm:2conv}The solution $\psi (u,z)= \sum a_{k, \ell}u^kz^{\ell}$ to \eqref{q1} with data \eqref{data2} satisfies
$$
\psi (u,0) = \sum_{k = 0}^{\infty} a_{k, 0} u^k
$$
with $a_{0,0} = 1$, $a_{1,0} = \tfrac{1}{2}(\al + \be )^2$, and  
$$
a_{k,0} = \frac{(-1)^{k+1}}{2 k!} \, (\al - \be )^2 (\al + \be )^2Q_k(\al , \be ) \qquad \forall k \geq 2,
$$
where 
$$
Q_k = \frac{(k-2)!}{2^{k-2}} \, \sum_{j = 1}^{k-1}\frac{(2j-1)!\, (2k-2j-1)!}{(j-1)!^2(k-j-1)!^2} \al^{2k-2j-2} \be^{2j-2}\,,
$$
is a symmetric polynomial in $\al$ and $\be$, homogeneous of degree $2k-4$, whose sum of coefficients of the $\al^{2k-2j-2} \be^{2j-2}$ $(j = 1, \ldots k-1)$ is equal to $2^{k-3}k!$.  Furthermore, the series converges aboslutely for all $u$ satisfying $|u| \leq 1/2\mu^2$, where $\mu = \max\{ |\al|, |\be |\}$.   
\end{theorem}

Note that, if we set $\al = 0$, so that $c_2 = 0$, then the radius of convergence coincides with that obtained in \S\ref{sec:one} for one arbitrary constant.  

\begin{proof}  The expression for $a_{k, 0}$ is obtained by recurrence, together with the aid of the software MAPLE.  The expression for the sum of the coefficients may be obtained by the use of generating functions. Specifically,  we show that 
$$
S_k:= \sum_{j=1}^{k-1} \frac{(2j-1)!}{(j-1)!^2} \frac{(2k-2j-1)!}{(k-j-1)!^2} = 2^{2k-5}k(k-1)\,,
$$
for each $k \geq 2$, which will establish the assertion of the theorem.  Indeed
\begin{eqnarray*}
S_k & = & \sum_{j = 1}^{k-1}(2j-1)(2k-2j-1)\left( \begin{array}{c} 2j-2 \\ j - 1 \end{array} \right) \left( \begin{array}{c} 2k-2j-2 \\ k- j - 1 \end{array} \right) \\ 
 & = & \sum_{j = 0}^m(2j+1)(2m-2j+1)\left( \begin{array}{c} 2j \\ j  \end{array} \right) \left( \begin{array}{c} 2m-2j \\ m- j  \end{array} \right)
\end{eqnarray*}  
where we set $m = k-2$.  However, this is the coefficient $c_m$ in the Cauchy product $\sum_{m=0}^{\infty} c_mt^m = S(t)^2$, where
$$
S(t) = \sum_{j=0}^{\infty} (2j+1) \left( \begin{array}{c} 2j \\ j \end{array} \right) t^j\,.
$$
But
$$
S(t) = 2\sum_{j=0}^{\infty} (j+1)\left( \begin{array}{c} 2j \\ j \end{array} \right) t^j - \sum_{j=0}^{\infty} \left( \begin{array}{c} 2j \\ j \end{array} \right) t^j\,,
$$
where
$$
\sum_{j=0}^{\infty} (j+1)\left( \begin{array}{c} 2j \\ j \end{array} \right) t^j = \frac{\dd}{\dd t}\sum_{j=0}^{\infty} \left( \begin{array}{c} 2j \\ j \end{array} \right) t^{j+1} = \frac{\dd}{\dd t} \frac{t}{\sqrt{1-4t}} = \frac{1-2t}{(1-4t)^{3/2}}\,.
$$
It follows that
$$
S(t) = \frac{2-4t}{(1-4t)^{3/2}} - \frac{1}{\sqrt{1 - 4t}} =  \frac{1}{(1-4t)^{3/2}}
$$
and so by Newton's binomial formula \eqref{nbf}
$$
S(t)^2 = \frac{1}{(1-4t)^3} =  \sum_{m=0}^{\infty} \left( \begin{array}{c} m+2 \\ 2 \end{array} \right) (4t)^m = \sum_{m = 0}^{\infty} 2^{2m-1}(m+2)(m+1)t^m\,.
$$
On replacing $m$ by $k-2$, this gives the required formula for $S_k$.  
   For the radius of convergence, we use the root test for a power series.  

The absolute value 
$$
|a_{k,0}u^k| = \frac{|\al - \be |^2|\al + \be |^2}{2k!} |Q_k(\al , \be )| \, |u|^k\,,
$$
can be estimated as follows.  By symmetry in $\al$ and $\be$, it is no loss of generality to suppose that $|\al / \be |\leq 1$.  If we write
$$
Q_k(\al , \be ) = b_0\al^{2k-4} + b_1 \al^{2k-6}\be^2 + \cdots + b_{k-2}\be^{2k-4}\,,
$$
Then
\begin{eqnarray*}
|Q_k(\al , \be )| & \leq  &  |\be |^{2k-4} \left\{ b_0\left| \frac{\al}{\be}\right|^{2k-4} + b_1\left| \frac{\al}{\be}\right|^{2k-6} + \cdots + b_{k-3}\left|\frac{\al}{\be}\right|^2 +  b_{k-2}\right\} \\
 & \leq & |\be |^{2k-4} (b_0 + b_1 + \cdots b_{k-2}) = |\be |^{2k-4}2^{k-3}k!\,,
\end{eqnarray*}
by the formula for the sum of the coefficients.  
We thus obtain
$$
\lim_{k \ra \infty} |a_{k,0}u^k| = 2|\be |^2|u| \lim_{k \ra \infty}\left| \frac{(\al - \be )^2(\al + \be )^2}{2^4\be^4}\right|^{1/k}
$$
which, for convergence, we require to be $<1$.  If $\al = \be$, then we have convergence for all $u$, otherwise, a sufficient condition is that $2|\be |^2 |u| < 1$.  
\end{proof}

We conclude, by considering the $1$-complex parameter family of solutions defined by $\al =\be$, which has the Hopf map as a special case $(\al = \be = - \ii$).  In this case
$$
\phi (x,y,z) = \frac{1}{x- \ii y}\{ \al^2(x^2+y^2) + (1+ \al z)^2\}\,,
$$
and the fibres are given by equations
$$
\al^2(x^2+y^2+z^2) + 2 \al z - 2 \eta (x-\ii y) + 1 = 0
$$
for a variable parameter $\eta \in \CC$.  If we define the complex $3$-vector
$$
\xi := \frac{1}{\al^2}(-\eta , \ii \eta , \al )
$$
then we see that the fibres are all circles with centre $- \Re \xi$  lying in the plane with normal $\Im \xi$.  The family of maps is a continuous deformation of the Hopf map, with limiting case given when $\al$ is real, when the circles all pass through the point $x = y = 0$, $z = - 1/\al$ and lie in planes containing the $z$-axis -- a so-called \emph{bouquet} of circles.

\appendix
\section{Proof of Proposition \ref{prop:fund}} \label{app:calcs}  We make use of the following lemma for the partial derivative of a product, whose proof is an exercise in induction.

\begin{lemma} \label{lem:der-prod}  For smooth functions $f$ and $g$ 
$$
\forall \, k, \ell \geq 0,\ \frac{\pa^{k + \ell} (fg)}{\pa u^k \pa z^{\ell}} = \sum_{i = 0}^k\sum_{j = 0}^{\ell} \left( \begin{array}{c} k \\ i \end{array} \right)  \left( \begin{array}{c} \ell \\ j \end{array} \right)  \frac{\pa^{k + \ell - i - j}f}{\pa^{k - i}u\,\pa^{\ell - j} z} \, \frac{\pa^{i + j}g}{\pa u^i \pa z^j}\,.
$$  
\end{lemma}

We now compute \eqref{der-eq}.  Suppose that $n = k + \ell$ with $k \geq 1$ and $\ell \geq 0$.  First we have
$$
\frac{\pa^n}{\pa u^k\pa z^{\ell}} (\psi \psi_u) = \sum_{i = 0}^k\sum_{j = 0}^{\ell} \left( \begin{array}{c} k \\  i \end{array} \right)  \left( \begin{array}{c} \ell \\  j \end{array} \right) \pa_u^{k-i}\pa_z^{\ell - j} \psi \cdot \pa_u^{i+1}\pa_z^j\psi\,,
$$
which, at $(u,z) = (0,0)$, gives
$$
\sum_{i = 0}^k\sum_{j = 0}^{\ell} \left( \begin{array}{c} k \\  i \end{array} \right)  \left( \begin{array}{c} \ell \\  j \end{array} \right) \psi_{k-i, \ell - j} \psi_{i+1, j}\,.
$$
For the next term we have
\begin{eqnarray*}
\frac{\pa^n}{\pa u^k\pa z^{\ell}} (u\psi_u^2) & =&  \sum_{i = 0}^k\sum_{j = 0}^{\ell} \left( \begin{array}{c} k \\  i \end{array} \right)  \left( \begin{array}{c} \ell \\  j \end{array} \right) \pa_u^{i}\pa_z^{ j} (u) \cdot \pa_u^{k-i}\pa_z^{\ell - j}(\psi_u^2) \\
 & = & u \pa_u^k\pa_z^{\ell} (\psi_u^2) + k \pa_u^{k-1}\pa_z^{\ell} (\psi_u^2) \\
& = & u \pa_u^k\pa_z^{\ell} (\psi_u^2) + k \sum_{i = 0}^{k-1}\sum_{j=0}^{\ell} \left( \begin{array}{c} k-1 \\ i \end{array} \right) \left( \begin{array}{c} \ell \\  j \end{array} \right) \pa_u^{k-i}\pa_z^{\ell - j} \psi \cdot \pa_u^{i+1}\pa_z^j\psi\,.
\end{eqnarray*}
At $(u,z) = (0,0)$ this gives
$$
k  \sum_{i = 0}^{k-1}\sum_{j=0}^{\ell} \left( \begin{array}{c} k-1 \\ i \end{array} \right) \left( \begin{array}{c} \ell \\  j \end{array} \right) \psi_{k-i, \ell - j} \psi_{i+1, j}\,.
$$
Note that for $i \leq k-1$, we have
$$
k\left( \begin{array}{c} k-1 \\ i \end{array} \right) \pm \left( \begin{array}{c} k \\ i  \end{array} \right) =  ( k-i \pm 1) \left( \begin{array}{c} k \\ i \end{array} \right)\,,
$$
which is also defined for $i = k$ and gives the correct term.  Finally,  
\begin{eqnarray*}
\frac{\pa^n}{\pa u^k\pa z^{\ell}} (\tfrac{1}{2}\psi_z^2) & = & \frac{\pa^n}{\pa u^{k-1}\pa z^{\ell}} (\psi_z\psi_{uz}) \\
 & = & \sum_{i = 0}^{k-1}\sum_{j=0}^{\ell}   \left( \begin{array}{c} k-1 \\ i \end{array} \right) \left( \begin{array}{c} \ell \\  j \end{array} \right) \pa_u^{k-i-1}\pa_z^{\ell - j + 1} \psi \cdot \pa_u^{i+1}\pa_z^{j+1}\psi\,.
\end{eqnarray*}
At $(u,z) = (0,0)$ this gives
$$
\sum_{i = 0}^{k-1}\sum_{j=0}^{\ell}   \left( \begin{array}{c} k-1 \\ i \end{array} \right) \left( \begin{array}{c} \ell \\  j \end{array} \right) \psi_{k - i - 1, \ell - j + 1}\psi_{i+1, j+1}\,.
$$
On combining the various terms, \eqref{eq1} is established.  

\section{The recurrence \eqref{rec1}}\label{app:calcs2}  
We wish to show that the recurrence \eqref{rec1} is solved by \eqref{sol-rec} (with $f(0) = -1$).   Separating the terms which contain $f(0)$, \eqref{rec1} becomes 
\begin{eqnarray*}
(k+1)f(k+1) - (3k-1)f(k) & = & \sum_{m=1}^{k-1} \big\{ (m+2)(k-m)f(m+1) f(k-m) \\
 & & \qquad  + \tfrac{1}{2}(2m-1)(2k-2m-1)f(m)f(k-m)\big\}\,.
\end{eqnarray*}
However, the solution \eqref{sol-rec} is also characterized by the recurrence relation $f(k+1) = 3(2k-1)f(k) / (k+2)$ for $k\geq 1$ with $f(1) = 1/2$, so we may use this to rewrite the second term in the above sum, to obtain the equivalent relation  
\begin{equation}\label{rec3}
(k+1)f(k+1) - (3k-1)f(k) 
 =  \tfrac{1}{6} \sum_{m=1}^{k-1} (m+2)\big( 8(k-m)  -1\big)     f(m+1)f(k-m)
\end{equation}
If we now substitute the hypothetical solution $\ds f(k) =   \frac{1}{k(k+1)}\left( \begin{array}{c} 2k-2 \\ k-1 \end{array} \right) \left(\frac{3}{2}\right)^{k-1}$, the right-hand side becomes
\begin{equation} \label{rec-reformed}
\tfrac{1}{6}\left(\tfrac{3}{2}\right)^{k-1} \sum_{m=1}^{k-1}     \frac{(8(k-m)  -1)}{(m+1)(k-m)(k-m+1)}\left( \begin{array}{c} 2m \\ m \end{array} \right)\left( \begin{array}{c} 2k-2m-2 \\ k-m-1 \end{array} \right)\,.
\end{equation}
We can now exploit generating functions to evaluate this sum.  

\begin{lemma}\label{lem:gen1}
\begin{eqnarray*}
\sum_{m=1}^{k-1} \frac{1}{(m+1)(k-m+1)} \left( \begin{array}{c} 2m \\ m \end{array} \right)\left( \begin{array}{c} 2k-2m-2 \\ k-m-1 \end{array} \right)  =  \hspace{3cm} \\
\frac{1}{12(k+2)} \left( \begin{array}{c} 2k+2 \\ k+1 \end{array} \right) 
+ \frac{1}{2(k+2)} \left( \begin{array}{c} 2k \\ k \end{array} \right) - \frac{1}{k+1} \left( \begin{array}{c} 2k-2 \\ k-1 \end{array} \right)  
\end{eqnarray*} 
\end{lemma}
\begin{proof}
From
$$
\frac{1}{\sqrt{1-4t}} = \sum_{k=0}^{\infty} \left( \begin{array}{c} 2k \\ k \end{array} \right) t^k
$$
we deduce that 
$$
S(t):= \sum_{k=0}^{\infty} \frac{1}{k+1} \left( \begin{array}{c} 2k \\  k \end{array} \right) t^{k+1} = \int\frac{\dd t}{\sqrt{1-4t}} = - \frac{\sqrt{1-4t}}{2} + \frac{1}{2}\,.
$$
Similarly
$$
R(t):= \sum_{k=0}^{\infty} \frac{1}{k+2} \left( \begin{array}{c} 2k \\ k \end{array} \right) t^{k+2} = \int \frac{t\dd t}{\sqrt{1-4t}} = - \frac{\sqrt{1-4t}}{8} + \frac{(1-4t)^{3/2}}{24} + \frac{1}{12}\,.
$$
Then
$$
S(t)R(t) = t^3 \sum_{k=0}^{\infty} c_kt^k \quad {\rm where} \quad c_k = \sum_{m=0}^k\frac{1}{(m+1)(k-m+2)} \left( \begin{array}{c} 2m \\ m \end{array} \right) \left( \begin{array}{c} 2k-2m \\ k -m \end{array} \right)
$$
It now follows that
$$
S(t)R(t) = tR(t) + t^2 \sum_{k=1}^{\infty} \left\{ \sum_{m=1}^{k-1} \frac{1}{(m+1)(k-m+1)} \left( \begin{array}{c} 2m \\ m \end{array} \right) \left( \begin{array}{c} 2k-2m-2 \\ k-m-1 \end{array} \right) \right\} t^k
$$
On the other hand
\begin{eqnarray*}
R(t)(S(t) - t) & = & \frac{1}{48} \Big( 4 - 8t - 16t^2 - (5-6t)\sqrt{1 - 4t} + (1 -2t)(1-4t)^{3/2}\Big) \\
 & = & \frac{t^2}{12} \sum_{k=1}^{\infty} \Bigg\{\frac{1}{k+2} \left( \begin{array}{c} 2k+2 \\ k+1 \end{array} \right) 
+ \frac{6}{k+2} \left( \begin{array}{c} 2k \\ k \end{array} \right) - \frac{12}{k+1} \left( \begin{array}{c} 2k-2 \\ k-1 \end{array} \right) \Bigg\} t^k 
\end{eqnarray*}
The formula for the sum now follows.  
\end{proof}

\begin{lemma}\label{lem:gen2}
\begin{eqnarray*}
\sum_{m=1}^{k-1} \frac{1}{(m+1)(k-m)(k-m+1)} \left( \begin{array}{c} 2m \\ m \end{array} \right)\left( \begin{array}{c} 2k-2m-2 \\ k-m-1 \end{array} \right) \qquad \qquad \\
  =  \frac{1}{6} \Bigg\{ -\frac{k-1}{(k+2)(2k+1)} \left( \begin{array}{c} 2k+2 \\ k+1 \end{array} \right) 
 + \frac{6(k-1)}{(k+1)(2k-1)} \left( \begin{array}{c} 2k \\ k \end{array} \right) \Bigg\} 
\end{eqnarray*} 
\end{lemma}

\begin{proof}
We proceed as in the last lemma, with $S(t)$ defined as in the above proof.  Define
$$
T(t) := \sum_{k=1}^{\infty} \frac{1}{k(k+1)} \left( \begin{array}{c} 2k-2 \\ k-1 \end{array} \right) t^{k+1}\,.
$$
Then 
$$
T''(t) = \sum_{k=1}^{\infty}  \left( \begin{array}{c} 2k-2 \\ k-1 \end{array} \right) t^{k-1} = \sum_{k=0}^{\infty}  \left( \begin{array}{c} 2k \\ k \end{array} \right) t^k = \frac{1}{\sqrt{1-4t}}\,,
$$
from which we deduce that
$$
T(t) = - \frac{1}{12} + \frac{t}{2} + \frac{1}{12} (1-4t)^{3/2}\,.
$$
As in the proof of Lemma \ref{lem:gen1}, we find that
$$
t^2 \sum_{k=1}^{\infty} \left\{ \sum_{m=1}^{k-1} \frac{1}{(m+1)(k-m)(k-m+1)} \left( \begin{array}{c} 2m \\ m \end{array} \right)\left( \begin{array}{c} 2k-2m-2 \\ k-m-1 \end{array} \right) \right\} t^k = T(t)(S(t) - t)
$$
and the formula follows. 
\end{proof} 

Combining the two lemmas allows us to evaluate \eqref{rec-reformed}, which in turn confirms that \eqref{rec3} is indeed solved by \eqref{sol-rec}.

\end{document}